\documentclass[12pt]{amsart}
\usepackage[american]{babel}
\usepackage{amsmath}
\usepackage{latexsym}
\usepackage{amssymb}
\usepackage[shortlabels]{enumitem}
\usepackage{tikz}
\usetikzlibrary{matrix,arrows,decorations.pathmorphing}
\usepackage{tikz-cd}
\usepackage[noadjust]{marginnote}
\usepackage{mathtools}
\usepackage{comment}
\DeclareMathAlphabet\mathbfcal{OMS}{cmsy}{b}{n}
\theoremstyle{plain}
\newtheorem{Thm}[subsection]{Theorem}
\newtheorem{Cor}[subsection]{Corollary}

\newtheorem{Lem}[subsection]{Lemma}

\theoremstyle{definition}

\newtheorem{Rem}[subsection]{Remark}
\newtheorem{Def}[subsection]{Definition}
\newtheorem{Prob}[subsection]{Problem}

\renewcommand{\phi}{\varphi}

\newcommand{\NN}{\mathbb{N}}
\newcommand{\id}{\mathrm{id}}
\newcommand{\pr}{\mathrm{pr}}

\renewcommand{\emptyset}{\varnothing}

\newcounter{mynote}

\begin{document}

\title{\bf{Fibrations and Coset Spaces for Locally Compact Groups}}
\author{Linus Kramer and Raquel Murat Garc\'ia}
\address{Linus Kramer, Raquel Murat Garc\'ia\newline\indent
Mathematisches Institut, Universit\"at M\"unster
\newline\indent
Einsteinstr. 62, 
48149 M\"unster,
Germany}

\begin{abstract}
    Let $G$ be a topological group and let $K,L\subseteq G$ be closed subgroups, with $K\subseteq L$.
    We prove that if $L$ is a locally compact pro-Lie group, then the map $q:G/K\to G/L$ is a 
    Serre fibration, and a Hurewicz fibration if $G$ is paracompact. As an application of this, we obtain two older results by Skljarenko, Madison and Mostert.
\end{abstract}

\maketitle

Corresponding author: Linus Kramer


\section{Introduction}

Suppose that $G$ is a topological group and that $K,L\subseteq G$ are closed subgroups
with $K\subseteq L$. If we endow the coset spaces $G/K$ and $G/L$ with the quotient topologies,
we have a natural continuous and open $G$-equivariant map 
\[
q:G/K\to G/L
\]
and the question is how $q$ behaves from the viewpoint of homotopy theory: is $q$ a locally trivial bundle? Is it a fibration?
If the natural map $G\to G/L$ admits local sections, then $q$ is indeed a locally trivial bundle, see Lemma \ref{local section gives trivial bundle} below.
In particular, $q$ is a locally trivial bundle and a Hurewicz fibration if $G$ is a Lie group.
We cannot expect such a result for general locally compact groups, as the following example shows.

\medskip\noindent\textbf{Example}
{
   Put $H=\mathrm{SU}(2)$ and $G=H^\NN$. Then $G$ is a compact connected and locally connected group.
   The center of $G$ is
   the compact totally disconnected group $L=\{\pm1\}^\NN$. The natural map $q:G\to G/L$ 
   does not admit local sections, since otherwise $G$ would be locally homeomorphic to the space
   $G/L\times L$, which is not locally connected. In particular, $q$ is not a locally trivial bundle.}

   \medskip
In the case that $G$ is a compact group, Madison and Mostert proved in 1969 \cite{Madison} that $q:G/K\to G/L$ is always a Hurewicz fibration (Madison attributes the proof to Mostert).
This applies in particular to the example above. A modern proof for their result is given in Theorem 2.8 in \cite{HofmannKramer}.
Already in 1963, Skljarenko \cite{Skljarenko} had proved that $q$ is a Hurewicz fibration if $G$ is a locally compact group.
Apparently his results were not widely noticed at that time.
In his 1970 PhD thesis \cite{Wigner}, Wigner claimed that for every topological group $G$, the map $q:G\to G/L$ is a fibration, provided that $L\subseteq G$ is a locally compact subgroup. However, his proof appears to have gaps. We comment on this at the end of our article.
Our main result is as follows.

\medskip\noindent\textbf{Theorem}
{\em Let $G$ be a topological group and let $K,L\subseteq G$ be closed subgroups,
with $K\subseteq L$. If $L$ is a locally compact pro-Lie group, then the map
\[
q: G/K\to G/L
\]
is a Serre fibration. If $G$ is paracompact, then $q$ is a Hurewicz fibration.}

\medskip
As a consequence of this result, we recover the results by Skljarenko, Madison and Mostert.
Our proof follows ideas similar to the proofs in \cite{Madison} and  \cite{HofmannKramer},
and it relies on Palais' Slice Theorem \cite{Palais} and an important result
by Antonyan \cite{Antonyan}. It is rather
different from Skljarenko's proof \cite{Skljarenko}, which uses 
well-ordered chains of Lie quotients.

\subsection*{Acknowledgements.}
We thank the referees for several helpful remarks which improved our presentation.
This work is based on the second author's MSC thesis \cite{Murat}.
Both autors acknowledge support from  Germany’s Excellence Strategy 
EXC 2044-390685587, Mathematics M\"unster: Dynamics-Geometry-Structure.

\section{Generalities on bundles and fibrations}

Throughout, all spaces and topological groups are assumed to be Hausdorff unless stated otherwise.
Other topological assumptions will be stated explicitly.

We recall some terminology about bundles. A \emph{bundle} over $B$ is a continuous surjective map
$p:E\to B$, where $E$ and $B$ are Hausdorff spaces. We call $E$ the \emph{total space}, $B$ the base space, and for $b\in B$
we denote by $E_b=p^{-1}(b)$ the \emph{fiber} over $b$. A \emph{morphism} between
bundles $p:E\to B$ and $p':E'\to B$ is a continuous map $f:E\to E'$ with $p'\circ f=p$:
\[
\begin{tikzcd}
    E \arrow{r}{f}\arrow{d}[swap]{p} & E' \arrow{d}{p'} \\
    B \arrow[equal]{r} & B.
\end{tikzcd}
\]
The morphisms of bundles over $B$ form a category in an obvious way.
A bundle is called \emph{trivial} if it is isomorphic to a bundle of the form
$\pr_1:B\times F\to B$.

 For $A\subseteq B$ we put $E_A=p^{-1}(A)$. The restriction $p|_{E_A}:E_A\to A$ is then a bundle over $A$. 
  The bundle 
 $E\to B$ is called \emph{locally trivial} if every $b\in B$ has a neighborhood
 $V\subseteq B$ such that the restriction $E_V\to V$ is trivial.

A \emph{section} over $A\subseteq B$ is a continuous map $s:A\to E_A$,
 with $p\circ s=\id_A$. A section over a neighborhood $U$ of a point $b\in B$ is called a  \emph{local section}.

Let $p:E\to B$ be a bundle and let $Z$ be a topological space. The bundle has the 
\emph{homotopy lifting property (HLP)} with respect to $Z$ if the following holds:
given any two continuous maps $f:Z\times[0,1]\to B$ and $\tilde f_0:Z\times\{0\}\to E$ such that 
$p(\tilde f_0(z,0))=f(z,0)$, there exists a continuous extension $\tilde f:Z\times[0,1]\to E$
of $\tilde f_0$ with $p\circ \tilde f=f$.  This situation is illustrated by the commutative diagram
\[
\begin{tikzcd}
    Z\times\{0\} \arrow{r}{\tilde f_0}\arrow[hook]{d} & E \arrow{d}{p}\\
    Z\times[0,1] \arrow{r}{f} \arrow[dashed]{ru}{\tilde f}&B .
\end{tikzcd}
\]

A bundle which has the HLP for every space $Z$ is called a \emph{Hurewicz fibration}.
It is called a \emph{Serre fibration} if it has the HLP for every cube $[0,1]^m$.
In this case, it has the HLP with respect to every CW complex, cp.~\cite{Bredon} Theorem 6.4, Chapter VII.
Every locally trivial bundle is a Serre fibration; such a bundle is
a Hurewicz fibration if $B$ is paracompact, see \cite{Dugundji} Theorem 4.2, Chapter XX.

\section{Quotients by locally compact groups}

Suppose that $G$ is a topological group, and that $K,L\subseteq G$ are closed subgroups
with $K\subseteq L$. If we endow $G/K$ and $G/L$ with the quotient topologies with respect to
the canonical maps $G\to G/K$ and $G\to G/L$, then the canonical map $q:G/K\to G/L$
is a continuous open $G$-equivariant map whose fibers are homeomorphic to $L/K$.

We could not find a reference for this fact, so we supply a proof.
We consider the commutative diagram
\[
\begin{tikzcd}
    L \arrow[hook]{r} \arrow{d}[swap]{p_1} & G \arrow[equal]{r} \arrow{d}[swap]{p_2} & G \arrow{d}[swap]{p_3} \\
    L/K \arrow[hook]{r}{i}& G/K\arrow{r}{q} & G/L .
\end{tikzcd}
\]
The vertical maps $p_1,p_2,p_3$ are quotient maps, and $L$ carries the subspace topology.
This shows already that $i$ and $q$ are continuous. The maps $p_1,p_2,p_3$ are also open,
cp.~Lemma~6.2 in \cite{Stroppel}, and the quotients are Hausdorff, cp.~Proposition 6.6 in \emph{loc.cit}.
If $U\subseteq G/K$ is open, then $q(U)=p_3(p_2^{-1}(U))$ is also open, hence $q$
is an open map. We claim next that $i\circ p_1$ is an open map onto its image. 
Let $W\subseteq G$ be open and put $V=L\cap W$. Then $VK/K=(WK/K)\cap (L/K)$, i.e. $i(p_1(V))=p_2(W)\cap L/K$.
Hence $i\circ p_1$ is an open map onto its image, and thus $i$ maps $L/K$ homeomorphically onto its image in $G/K$.
Finally, all fibers of $q$ are homeomorphic under the $G$-action on $G/K$.

A \emph{Lie group} is a group which is at the same time a differentiable manifold,
such that the multiplication and inversion maps are smooth. We do not assume that Lie groups
are second countable; indeed, the identity component $G^\circ$ of a Lie group $G$ is automatically
open and second countable, and $G$ is metrizable and paracompact, cp. Proposition~9.1.15 and Theorem~10.3.25 in \cite{HilgertNeeb}. 
We recall also that closed subgroups of Lie groups
are again Lie groups. We refer to the excellent book \cite{HilgertNeeb} for generalities about Lie groups.
\begin{Lem}\label{local section gives trivial bundle}
    Let $G$ be a topological group with closed subgroups $K\subseteq L\subseteq G$.
    Suppose that the canonical map $G\to G/L$ admits a local section. Then $q:G/K\to G/L$
    is a locally trivial bundle. This holds in particular if $G$ is a Lie group.
\end{Lem}
\begin{proof}
    The group $G$ acts transitively on $G/L$ and equivariantly on $q:G/K\to G/L$. The existence of a local section near some point
    $x\in G/L$ therefore implies the existence of a local section near every point $y\in G/L$.
    Let $s:U\to G$ be a local section, where $U\subseteq G/L$ is a nonempty open subset.
    Then the map $h:U\times L/K\to W=q^{-1}(U)$ that maps $(u,\ell K)$ to $s(u)\ell K$
    is a homeomorphism, with inverse $w\longmapsto(q(w),s(q(w))^{-1}w)$, which makes the diagram
    \[
    \begin{tikzcd}
        U\times L/K \arrow{r}{h} \arrow{d}[swap]{\pr_1} & W \arrow{d}{q|_W} \\
        U \arrow[equal]{r} & U
    \end{tikzcd}
    \]
    commute.
    For the last claim, we note that $G\to G/L$ admits local sections if $G$ is a Lie group, cp.
    Corollary 10.1.11 in \cite{HilgertNeeb} or Theorem 3.58 in \cite{Warner}.
\end{proof}
We recall that a continuous surjective map is called \emph{proper} (or \emph{perfect})
if it is closed
and has compact fibers.
\begin{Lem}\label{ProperLemma}
    Let $G$ be a topological group with locally compact subgroups $K\subseteq H\subseteq G$.
    If $H/K$ is compact, then the canonical map $q:G/K\to G/H$ is proper.
\end{Lem}
\begin{proof}
    Every $h\in H$ has a compact neighborhood $D_h$ in $H$. Since $H/K$ is compact, there are finitely
    many elements $h_1,\ldots, h_k\in H$ such that $p(D_{h_1})\cup \cdots\cup p(D_{h_k})=H/K$,
    where $p:H\to H/K$ is the canonical map.
    We put $C=D_{h_1}\cup \cdots\cup D_{h_k}$ and we note that $H=KC$. 
    Suppose that $A\subseteq G/K$ is closed, with closed
    preimage $B$ in $G$. The preimage of $q(A)\subseteq G/H$ in $G$ is then the closed set 
    $BH=BC$, cp. Lemma~3.19 in \cite{Stroppel},
    hence $q(A)$ is closed. The fibers of $q$ are homeomorphic to $H/K$ and therefore compact.
\end{proof}

\begin{Lem}\label{Limit}
    Let $G$ be a topological group, let $(I,\leq)$ be a nonempty linearly ordered set and let 
    $(H_i)_{i\in I}$ be a family of locally compact subgroups of $G$, such that 
    $H_i\subseteq H_j$ 
    holds for all $i\leq j$. Put $H=\bigcap_{i\in I}H_i$.
    The natural maps $G/H_i\to G/H_j$ make up a projective system in the category of topological spaces and continuous maps.
    If the quotients $H_i/H$ are compact, then the natural map
    \[
    G/H\to \varprojlim G/H_i.
    \]
    is a homeomorphism.
\end{Lem}
\begin{proof}
    One model for the projective limit $\varprojlim G/H_i$ consists of all elements $(g_iH_i)_{i\in I}$ of $\prod_{i\in I}G/H_i$ with the property that
    $g_iH_j=g_jH_j$ holds for $i,j$ with $i\leq j$, cp. Appendix Two in \cite{Dugundji}. 
    Given such an element $(g_iH_i)_{i\in I}$, the collection of compact sets $\{g_iH_i/H\subseteq G/H\mid i\in I\}$
    has the finite intersection property and contains therefore a common element $gH\in G/H$. 
    Then $g_iH_i=gH_i$ holds for all $i$. 
    We consider the natural continuous map $f:G/H\to\varprojlim G/H_i$
    that maps $gH$ to $(gH_i)_{i\in I}$.
    This map $f$ has a two-sided inverse $t$ that maps $(gH_i)_{i\in I}$ to $gH=\bigcap_{i\in I}gH_i$.
    In particular, $f$ is a continuous bijection. We claim that $f$ is a closed map, and hence
    a homeomorphism.

    By Lemma~\ref{ProperLemma}, the maps $p_i:G/H\to G/H_i$ are proper.
    Therefore the product map
    $p:\prod_{i\in I}G/H\to\prod_{i\in I}G/H_i$ is also a proper
    map by Theorem 3.7.9 in \cite{Engelking}, and hence closed. The diagonal map $d:G/H\to \prod_{i\in I}G/H$
    is also closed, cp.~Corollary 2.3.21 in \emph{loc.cit}. Hence $f=p\circ d$ is closed.
\end{proof}
Let $G$ be a topological group, let $X$ be a space and suppose that the group $G$ acts (from the right) on the set $X$. 
We say that this action is a \emph{right transformation group} if the action map $X\times G\to X$ is continuous. For each element $x\in X$, we denote its stabilizer by $G_x=\{g\in G|xg=x\}$.
Left transformation groups are defined analogously.
We recall some terminology and results due to Palais~\cite{Palais}.
\begin{Def}
    Let $X\times G\to X$ be a (right) transformation group. For subsets $A,B\subseteq X$ we put
    $G_{A,B}=\{g\in G\mid (Ag)\cap B\neq\emptyset\}$. If $A$ and $B$ are open, then 
    $G_{A,B}$ is an open subset of $G$.
    We say that the action is a \emph{Cartan action} if $X$ is completely regular and if
    every point $x\in X$ has an open
    neighborhood $U$ such that $G_{U,U}$ has compact closure in $G$. This condition implies that
    $G$ is locally compact. 
\end{Def}
\begin{Lem}\label{CartanActionExample}
    Let $G$ be a topological group with a locally compact subgroup $L\subseteq G$,
    and let $N\unlhd L$ be a closed normal subgroup of $L$.
    Then the natural right action of $L/N$ on $G/N$ is a Cartan action.
\end{Lem}
\begin{proof}
    First of all, $G/N$ is completely regular, cp.~II.8.14 in \cite{HewittRoss}.
    Let $W\subseteq G$ be an open symmetric identity
    neighborhood in $G$ such that $WW\cap L$ has compact closure in $L$.
    For $g\in G$ we put $U=gWN/N\subseteq G/N$. 
    Then $U$ is an open neighborhood of $gN$. Let $\ell\in L$.
    If $U\ell N\cap U\neq\emptyset$,
    then $gw_1\ell n=gw_2$ holds for certain $w_1,w_2\in W$ and $n\in N$, whence 
    $\ell n\in WW$ and $\ell N\in WWN/N$. The right-hand side has compact closure in $L/N$.
\end{proof}

\begin{Def}
    Let $X\times H\to X$ be a transformation group and let $K\subseteq H$ be a closed subgroup.
    A subset $S\subseteq X$ is called a \textit{$K$-kernel} if there exists an $H$-equivariant map $f: SH\to K\backslash H$ 
    such that $f^{-1}(K)=S$. If in addition $SH$ is open in $X$, then $S$ is called a \textit{$K$-slice}. For $x\in X$,
    a \textit{slice at $x$} is an $H_x$-slice that contains $x$, where 
    $H_x$ denotes the stabilizer of $x$.
\end{Def}

Palais proved that for Cartan actions $X\times H\to X$ of Lie groups there exist 
$H_x$-slices for every $x\in X$, cp.~Theorem~2.3.3 in \cite{Palais}. 
Abels and Lütkepohl later considered the existence of slices in a somewhat different setting \cite{Abels}. 
We recall that a group action $X\times H\to X$ is called \emph{free} if $H_x=\{e\}$ holds for 
all $x\in X$.
As an immediate consequence of Palais' theorem applied to free actions, we get the following result.

\begin{Thm}[Palais]\label{PalaisThm}
    Let $X\times H\to X$ be a Cartan action of a Lie group $H$. If the action is free,
    then the orbit map $q:X\to X/H$ admits local sections. Moreover, $q$ is a locally trivial bundle.
\end{Thm}
\begin{proof}
    The map $q$ is an $H$-principal bundle in the terminology of~\cite{Palais}.
    By \emph{op.cit.} Theorem ~2.3.3, there exists a slice $S$ at every point $x\in X$.
    This means in our situation that there is a subset $S\subseteq X$ such that
    $SH\subseteq X$ is open, and a continuous equivariant map $f:SH\to H$
    with $S=f^{-1}(e)$. This implies that $f(xh)=h$ for all $x\in S$, $h\in H$.
    Hence the map 
    $S\times H\to SH$ that maps $(x,h)$ to $xh$ is a homeomorphism,
    with inverse $xh\mapsto(xhf(xh)^{-1},f(xh))$.
    
    The set $q(SH)=SH/H$ is open in $X/H$ and the restriction $S\times H\cong SH\to SH/H$ is a trivial bundle.
    Finally, the map $s:SH/H\to S$ that maps $xH$ to
    $x$ is continuous and hence a local section for $q$.
\end{proof}
\begin{Cor}\label{LocallyTrivial}
    Let $G$ be a topological group with closed subgroups $P\subseteq Q\subseteq G$.
    Assume that $N$ is a compact normal subgroup of $Q$, and that $Q/N$ is a Lie group.
    Then the bundle $p:G/PN\to G/Q$ is locally trivial.
    In particular, $p$ is a Serre fibration, and a Hurewicz fibration if $G$ is paracompact.
\end{Cor}
\begin{proof}
    The Lie group action $G/N\times Q/N\to G/N$ is free, hence the bundle
    $G/N\to (G/N)/(Q/N)\cong G/Q$ has local sections by Theorem~\ref{PalaisThm}.
    Let $s:U\to G/N$ be a section over an open set $U\subseteq G/Q$.
    We note that the map $G/N\times Q/PN\to G/PN$ that maps
    $(gN,qPN)$ to $gqPN$ is continuous, because arrow (1) in the diagram
    \[
    \begin{tikzcd}
        G\times Q\arrow{r}{\text{multiply}}\arrow{d}[swap]{(1)} & G\arrow{d} \\
        G/N\times Q/PN\arrow{r} & G/PN.
    \end{tikzcd}
    \]
    is a product of open maps and hence open.
    Then the map $h:U\times Q/PN\to G/PN$ that maps $(gQ,qPN)$ to $s(gQ)qPN$ is a trivialization
    of the bundle $G/PN\to G/Q$ over $U$.

    Every locally trivial bundle is a Serre fibration by Theorem~4.2 in Chapter XX of \cite{Dugundji}. If $G$ is paracompact, then
    $G/Q$ is also paracompact by Corollary~1.5 in \cite{Antonyan} and hence $G/PN\to G/Q$ is a Hurewicz fibration by Theorem~4.2 in 
    Chapter~XX
    of \cite{Dugundji}.
\end{proof}

\begin{Def}\label{ProLieDef}
    We call a topological group $L$ a \emph{locally compact pro-Lie group} if every identity neighborhood $U\subseteq L$ contains a compact normal subgroup $N$ such that $L/N$ is a Lie group. In this case, every closed subgroup $H\subseteq L$ is also 
    a locally compact pro-Lie group. Indeed, $H/H\cap N\cong HN/N\subseteq L/N$ is then a closed subgroup of the Lie group $L/N$,
    and therefore a Lie group.
    We refer to the excellent book \cite{HMPro} on pro-Lie groups, in 
    particular to Definition~3.25 and Remark 1.31.

    The solution of Hilbert's 
    5th Problem due to Iwasawa, Gleason, Yamabe, Montgomery and Zippin says that every locally compact group $L$ which is almost connected
    (i.e. $L/L^\circ$ is compact) is a locally compact pro-Lie group. Moreover, every locally compact group has an open almost connected subgroup which is a locally compact pro-Lie group.
    References are 
    p.~175 in \cite{MontgomeryZippin} or Theorem~6.0.11 in \cite{Tao}, and Theorem~4.4 in \cite{HofmannKramer}.
\end{Def}

\medskip
Now we get to our main result.

\begin{Thm}\label{MainThm}
    Let $G$ be a topological group and let $K, L \subseteq G$ be closed subgroups, with $K\subseteq L$. Assume that $L$ is a locally compact pro-Lie group.
    Then the bundle
    \[q:G/K\to G/L\]
    is a Serre fibration. If $G$ is paracompact, then $q$ is a
    Hurewicz fibration.
\end{Thm}
\begin{proof}
    Let $Z$ be a topological space and let $f:Z\times[0,1]\to G/L$
    and $\tilde f_0:Z\times\{0\}\to G/K$ be continuous maps such that 
    $q\circ \tilde f_0(z,0)=f(z,0)$. We need to study the lifting problem
     \[
    \begin{tikzcd}
        Z\times\{0\} \arrow{r}{\tilde f_0}\arrow[hook]{d}{i} & G/K \arrow{d}{p} \\
        Z\times [0,1]\arrow{r}{f}\arrow[dashed]{ur}{\tilde f}& G/L,
    \end{tikzcd}
    \]
    where either $Z=[0,1]^m$ is a cube, or $G$ is paracompact and $Z$ is any Hausdorff space.
    For this, we consider the collection $\mathcal H$ consisting of all pairs $(H,h)$,
    where $H\subseteq G$ is a closed subgroup with $K\subseteq H\subseteq L$ 
    such that
    $H/K$ is compact, and $h:Z\times[0,1]\to G/H$ is a continuous map that
    solves an intermediate lifting problem, i.e. $h$ makes the following diagram commute
    \[
    \begin{tikzcd}
        Z\times\{0\} \arrow{r}{\tilde f_0}\arrow[hook]{dd}{i} & G/K\arrow{d} \\
        & G/H\arrow{d} \\
        Z\times [0,1]\arrow{r}{f}\arrow{ru}{h}& G/L ,
    \end{tikzcd}
    \]
    where the vertical arrows on the right are the natural maps.
    We turn $\mathcal H$ into a poset by declaring $(H,h)\leq(H',h')$
    if $H\supseteq H'$ and if the diagram
    \[
    \begin{tikzcd}
    & G/H' \arrow{d} \\
    Z\times[0,1] \arrow{r}{h} \arrow{ru}{h'} & G/H
    \end{tikzcd}
    \]
    commutes. Our aim is now 
    to show that there is an element $(K,\tilde f)\in \mathcal H$.

    \medskip\noindent\emph{Claim 1. The set $\mathcal H$ is nonempty.}\\
    Let $N\unlhd L$ be a compact subgroup such that $L/N$ is a Lie group,
    and put $H=KN$.  We note that $H/K\cong N/N\cap K$ is compact, and
    we apply Corollary~\ref{LocallyTrivial}
    with $(K,L,N)=(P,Q,N)$. 
    Thus $G/H\to G/L$ is a Serre fibration, and a Hurewicz fibration if $G$
    is paracompact.
    Hence $f$ has a lift $h$ such that $(H,h)\in\mathcal H$,
    provided that $Z=[0,1]^m$, or for a general $Z$ if $G$ is paracompact. 
    
    \medskip\noindent\emph{Claim 2. If $(H,h)\in\mathcal H$ with $H\neq K$, then there exists 
    $(H',h')\in\mathcal H$ with $(H,h)<(H',h')$.}\\
    Since $H\subseteq L$ is closed, $H$ is a locally compact pro-Lie group, cp.~Definition~\ref{ProLieDef}.
    Since $H/K$ is Hausdorff, there is an open identity neighborhood $W\subseteq H$ such that $WK\neq H$.
    Let $N\subseteq W$ be a compact normal subgroup of $H$ such that $H/N$ is a Lie group.
    We put $H'=KN$.
    Applying Corollary~\ref{LocallyTrivial} to $(K,H,N)=(P,Q,N)$, we see that 
    $G/H'\to G/H$ is a Serre fibration, and a Hurewicz fibration if $G$ is paracompact.
    Hence $h$ has a lift $h':Z\times[0,1]\to G/H'$ such that $(H',h')\in\mathcal H$,
    provided that $Z=[0,1]^m$, or for a general $Z$ if $G$ is paracompact. 
    The diagram
    \[
    \begin{tikzcd}
    & G/H' \arrow{d} \\
    Z\times[0,1] \arrow{r}{h} \arrow{ru}{h'} & G/H
    \end{tikzcd}
    \]
    commutes, whence $(H,h)<(H',h')$.

    \medskip\noindent\emph{Claim 3. The poset $(\mathcal H,\leq)$ is inductive, i.e. every linearly ordered subset has an upper bound.}\\
    Let $(I,\leq)$ be a linearly ordered set, with $(H_i,h_i)\in\mathcal H$,
    such that $(H_i,h_i)\leq (H_j,h_j)$ holds for $i\leq j$. Put $H=\bigcap_{i\in I} H_i$.
    From the fact that $H_i/K$ is compact and from the maps $H/K\to H_i/K\to H_i/H$
    we see that $H_i/H$ and $H/K$ are also compact.
    By Lemma~\ref{Limit}, $G/H$ is the projective limit of the projective system formed by the maps $G/H_j\rightarrow G/H_i$.
    From the universal property of the projective limit we obtain for each $i\in I$ a 
    commutative diagram
    \[
    \begin{tikzcd}
        Z\times\{0\} \arrow{rr}{\tilde f_0}\arrow[hook]{dd} && G/K \arrow{d}{p_i} \arrow{ld}[swap]{p}\\
        & G/H \arrow{r} & G/H_i \arrow{d} \\
        Z\times[0,1]\arrow{rr}[swap]{f}\arrow{rru}[swap]{h_i}\arrow{ru}{h} && G/L .
    \end{tikzcd}
    \]
    Here $p_i:G/K\to G/H_i$ is the natural map, $p=\varprojlim p_j$ and $h=\varprojlim h_j$.
    The natural map $\tilde p:G/K\to G/H$ also makes the upper-right triangle commute. Therefore $p=\tilde p$, whence $(H,h)\in\mathcal H$.

    \medskip
    Now we finish the proof. Since $\mathcal H$ is a nonempty inductive poset by Claim~1 and Claim~3, it has maximal
    elements by Zorn's Lemma. Let $(H,h)$ be a maximal element in $\mathcal H$. By Claim 2, $H=K$ and hence $\tilde f=h$ solves the
    lifting problem.
\end{proof}
As a consequence, we recover a result due to Skljarenko, cp.~Theorem 15 in \cite{Skljarenko}.
We first note the following fact, which is used implicitly in \emph{op.cit.}
Suppose that $p:E\to B$ is a bundle, and that the base space is a coproduct
$B=\coprod_{s\in S}B_s$ (a disjoint union of closed and open subspaces $B_s$). 
Assume that for each $s\in S$, the space $E_{B_s}$
is also a coproduct, $E_{B_s}=\coprod_{t\in T}E_{s,t}$.
If each map $E_{s,t}\to B_s$ is a Hurewicz fibration, then $E\to B$ is a Hurewicz fibration.
Indeed, if we are given a lifting problem
\[
\begin{tikzcd}
    Z\times\{0\} \arrow{r}{f_0} \arrow[hook]{d}& E\arrow{d} \\
    Z\times[0,1] \arrow{r}{f}\arrow[dashed]{ru}{\tilde f} & B,
\end{tikzcd}
\]
we may decompose $Z$ into the disjoint closed and open subsets $Z_{s,t}\times\{0\}=f_0^{-1}(E_{s,t})$.
Then $f(Z_{s,t}\times[0,1])\subseteq B_s$ (because $[0,1]$ is connected
and $B_s$ is closed and open),
and it suffices to solve the lifting problems
\[
\begin{tikzcd}
    Z_{s,t}\times\{0\} \arrow{r}{f_0} \arrow[hook]{d}& E_{s,t}\arrow{d} \\
    Z_{s,t}\times[0,1] \arrow{r}{f}\arrow[dashed]{ru}{\tilde f_{s,t}} & B_s 
\end{tikzcd}
\]
for all $s,t$ to obtain a global solution $\tilde f$. The following is Theorem~15 in \cite{Skljarenko}.
In the case that $G$ is compact, the result is also due to Madison and Mostert \cite{Madison}, cp.~Theorem~2.8 in~\cite{HofmannKramer}.
\begin{Thm}[Skljarenko]
    Let $G$ be a locally compact group and let $K,L\subseteq G$ be closed subgroups such that $K\subseteq L$.
    Then the natural map $G/K\to G/L$ is a Hurewicz fibration.
\end{Thm}
\begin{proof}
    We essentially follow Skljarenko's proof~\cite{Skljarenko}.
    Let $H\subseteq G$ be an open subgroup which is a locally compact pro-Lie-group,
    cp.~Definition~\ref{ProLieDef}.
    We consider the double cosets $HgL\in H\backslash G/L$, which partition $G$ into disjoint
    open subsets. Since $G\to G/L$ is an open map, the image $HgL/L$ of $HgL$ in $G/L$ is
    open. For each $\ell\in L$ and $g\in G$ we have a commutative diagram
    \[
    \begin{tikzcd}
        H/H\cap g\ell K(g\ell)^{-1} \arrow{r}{(3)}\arrow{d}{(1)} & H/H\cap gLg^{-1}\arrow{d}{(2)} \\
         Hg\ell K/K \arrow{r}{(4)} & Hg L/L , \\
    \end{tikzcd}
    \]
    where the two vertical arrows are $H$-equivariant homeomorphisms.
    The horizontal arrow (3) is a Hurewicz fibration by Theorem~\ref{MainThm},
    because $H$ is a locally compact pro-Lie group.
    Hence (4) is also a Hurewicz fibration.
    Since $G/K=\bigcup \{Hg\ell K/K\mid g\in G,\ell\in L\}$
    and $G/L=\bigcup\{ HgL/L\mid g\in G\}$, the claim follows by the remarks preceeding this theorem.
\end{proof}

\begin{Rem}
As we mentioned in the introduction, Wigner claims in Prop.~2 in~\cite{Wigner} that $G\to G/L$ is a fibration if
$L$ is a locally compact subgroup of the topological group $G$. However, there are several issues with the proof.

The author argues first that the result is true if $L$ is compact (and tacitly uses a deep result due to Gleason
on the existence of local sections for compact Lie transformation groups). So far the proof follows 
ideas similar to Madison and Mostert \cite{Madison} and appears to be correct, albeit very brief.

For general locally compact groups, Wigner argues that $L$ has an open subgroup $L'$ which has a compact normal subgroup 
$L''\unlhd L'$ such that $L'/L''$ is a Lie group.
Then he uses a result due to Serre in order to show that $G/L''\to G/L'$ is a locally trivial bundle.
However, Serre's proof is only valid for compact Lie group actions. This part could be fixed by 
invoking Palais' results \cite{Palais} instead. Once this is done, it follows that $G\to G/L'$ is a fibration, since it is
the composite of two fibrations $G\to G/L''\to G/L'$.

The last claim is that the map $G/L'\to G/L$ is a covering map since the fiber $L/L'$ is discrete.
However, no argument is given for this claim, so we put it as an open question.
\end{Rem}
\begin{Prob}
    Let $G$ be a topological group and let $K,L\subseteq G$ be locally compact subgroups, with $K
    \subseteq L$ open in $L$. Is $G/K\to G/L$ a fibration?
\end{Prob}
If the answer to the problem is affirmative, we may drop the assumption that $L$ is a pro-Lie group in
Theorem~\ref{MainThm}.



\end{document}